\documentclass[a4paper,10pt]{article}
\usepackage{amsmath,amsthm,amsfonts}
\newtheorem*{thm*}{Theorem}

\newtheorem{thm}{Theorem}
\newtheorem{lem}[thm]{Lemma}
\newtheorem{cor}[thm]{Corollary}
\newtheorem{prop}[thm]{Proposition}
\theoremstyle{definition}
\newtheorem{defi}[thm]{Definition}
\theoremstyle{remark}
\newtheorem{rem}[thm]{Remark}

\numberwithin{equation}{section}
\numberwithin{thm}{section}
\newcommand{\R}{\mathbb{R}}
\newcommand{\Rn}{\mathbb{R}^n}
\newcommand{\Rnk}{\mathbb{R}^{n+k}}
\newcommand{\Rnu}{\mathbb{R}^{n+1}}

\newcommand{\Hnk}{\mathbb{H}^{n+k+1}}
\newcommand{\N}{\mathbb{N}}
\newcommand{\cS}{\mathcal{S}}
\newcommand{\cC}{\mathcal{C}}
\newcommand{\cA}{\mathcal{A}}

\newcommand{\cB}{\mathcal{B}}
\newcommand{\Dp}{{\mathcal{D}'}}

\newcommand{\Vo}{V_\Omega}
\newcommand{\Voe}{V_{\Omega,\varepsilon}}
\newcommand{\bV}{\bar{V}}
\newcommand{\bVe}{\bar{V}_\varepsilon}

\newcommand{\bD}{\bar{D}}

\newcommand{\bphi}{\bar{\varphi}}

\newcommand{\ex}{\mathrm{e}}
\newcommand{\Vol}{\textup{Vol}}
\newcommand{\vol}{\textup{vol}}
\newcommand{\supp}{\mathrm{Spt}}

\newcommand{\hess}{D^2}
\newcommand{\bhess}{\bar{D}^2}
\newcommand{\bnabla}{\bar{\nabla}}
\newcommand{\tr}{\mathrm{tr}}
\renewcommand{\div}{\mathrm{div}}
\newcommand{\dd}[2]{\frac{\partial #1}{\partial #2}}
\begin{document}

\begin{center}
    \Large
    \textbf{Submanifolds, Isoperimetric Inequalities and Optimal Transportation}\\
    \normalsize
    \vspace{10mm}
    Philippe CASTILLON \\
    \small
    \textsc{i3m} (\textsc{umr cnrs} 5149), Univ. Montpellier II,
    34095 \textsc{Montpellier} Cedex 5, France \\
    \texttt{cast\symbol{64}math.univ-montp2.fr}
    \normalsize
\end{center}

  \vspace{10mm}

  \small
  \noindent\textbf{Abstract :} The aim of this paper is to prove isoperimetric
  inequalities on submanifolds of the Euclidean space using mass transportation
  methods. We obtain a sharp ``weighted isoperimetric inequality'' and a nonsharp
  classical inequality similar to the one obtained in \cite{Michael-Simon}.

  The proof relies on the description of a solution of the problem of
  Monge when the initial measure is supported in a submanifold and the
  final one supported in a linear subspace of the same dimension.

  \vspace{5mm}
  \noindent\textbf{R\'esum\'e :} Le but de cet article est de d\'emonter des in\'egalit\'es
  isop\'erim\'etriques sur les sous-vari\'et\'es de l'espace euclidien en utilisant des
  m\'ethodes de transport optimal de mesures. On obtient ainsi une ``in\'egalit\'e
  isop\'erim\'etrique \`a poids'' avec constante optimale et une in\'egalit\'e classique
  similaire \`a celle obtenue dans \cite{Michael-Simon}.

  La preuve repose sur la description d'une solution du probl\`eme de Monge entre
  une mesure initiale support\'ee par une sous-vari\'et\'e et une mesure finale
  support\'ee par un sous-espace de m\^eme dimension.

  \vspace{5mm}
  \noindent\textbf{Mathematics Subject Classifications (2000) :} 53C42, 51M16.

  \vspace{5mm}
  \noindent\textbf{Key words :} submanifold, isoperimetric inequality, Sobolev inequality,
  	optimal transportation.

  \vspace{10mm}
  \normalsize

%
%
%
%
\section*{Introduction}

The classical isoperimetric inequality of the Euclidean space states that,
for any regular domain $\Omega\subset\Rn$,
$$
n\omega_n^\frac{1}{n}\Vol(\Omega)^\frac{n-1}{n} \le \vol(\partial\Omega)
$$
with equality if and only if $\Omega$ is a ball ($\omega_n$ being the volume
of the unit ball). This inequality admits a lot of generalisation to other
geometries (cf. \cite{Osserman} for a classical survey, and \cite{Ros} for a
more recent one), and on the other hand, a natural question
is to find geometries that share the Euclidean isoperimetric inequality.
One of the class of riemanniann manifolds expected to satisfy this inequality
is the class of minimal submanifolds in Euclidean spaces, and more
generaly in Cartan-Hadamard manifolds.

In this setting, the existence of a positive
isoperimetric constant was proved by J.H. Michael and L.M. Simon in the more
general setting of arbitrary submanifolds (cf. \cite{Michael-Simon}) :
\itshape
there exist a positive constant $C_n$, depending only on $n$, such that
for any domain $\Omega$ in a $n$ dimensional submanifold of $\Rnk$
$$
C_n\Vol(\Omega)^\frac{n-1}{n} \le \vol(\partial\Omega)+n\int_\Omega|H|dv_M
$$
where $H$ is the mean curvature vector of $M$.
\upshape

This result was then extended to submanifolds in Cartan-Hadamard manifolds (cf.
\cite{Hoffman-Spruck} and \cite{Castillon}), but the question of the optimal
constant for this inequality is still an open problem, even for minimal
surfaces in $\R^3$ (cf. \cite{Choe1}, \cite{Choe-Gulliver1}, \cite{Choe-Gulliver2}
for partial results, and \cite{Choe2} for a survey on this question).

A way to prove the Euclidean isoperimetric inequality is to construct a map,
with fine geometric properties, which push forward the uniform measure on
$\Omega$ to the uniform measure on the unit ball : this has to be seen as a
way to compare the domain $\Omega$ to the model domain satisfying the equality case.
This approach was first used by M. Gromov using a map constructed by Knothe (cf.
for example \cite{Chavel} for the proof), and in the sequel we shall refer to such a mapping
 as a ``Knothe map''.

More recently, D. Cordero-Erausquin, B. Nazaret and C. Villani observed that
the solution of an optimal transportation problem between the two
measures could be used as a ``Knothe map'' (cf. \cite{Cordero-Nazaret-Villani}) :
a theorem by Y. Brenier states that, if $\mu$ is a probability measure on $\Rn$
that do not give mass to small sets (ie. sets with Hausdorff dimension less than
or equal to $n-1$) then, for any probability measure $\nu$, there exists a convex
function whose gradient push forward $\mu$ on $\nu$. This approach was also
used in \cite{Figalli-Ge} to get isoperimetric type inequalities in space form.

In the case of an $n$ dimensional submanifold of $\Rnk$ we would like to compare the
uniform measure on $\Omega$ with the model measure which is the uniform one on
the unit ball of $n$ dimensional subspace of $\Rnk$ ; however, we are precisely
in the case where Brenier's theorem does not hold as the first measure is supported
in a small set. The goal of this paper is to deal with the two
following questions : considering two measure in $\Rnk$ supported in
submanifold and in a linear subspace of the same dimension, what are the solutions
of the optimal transportation problem ? Do these solutions have fine geometric
properties to give isoperimetric inequalities on the submanifold ?

In the first section we recall the main results which will be used in the remainder
of the paper : the equivalence between isoperimetric and Sobolev inequalities,
existence and properties of the solution of the optimal transportation problem
in Euclidean space, and differentiability properties of convex functions.

In the second section we describe solutions of the mass transportation problem
between a measure supported in a submanifold and a measure supported in a linear
subspace. It is shown in particular that orthogonal projections play a natural
role in this problem.

The third section is devoted to the proof of the main theorem : using the optimal
map we can compare the uniform measure on a domain in a submanifold with
the model measure. We get the following sharp ``weighted isoperimetric inequality''
(cf. theorem \ref{thm-sobL1}) :
\begin{thm*}
  Let $i:M^n\to\Rnk$ be an isometric immersion, and let $E$ be a $n$-dimensional
  linear subspace of $\Rnk$. For any regular domain $\Omega\subset M$ we have
  $$
  n\omega_n^\frac{1}{n}\Bigl(\int_\Omega J_E^\frac{1}{n-1}dv_M\Bigr)^\frac{n-1}{n}
   \le \vol(\partial\Omega) + n\int_\Omega|H|dv_M,
  $$
  where $H$ is the mean curvature vector of the immersion, and $J_E$ is the
  absolute value of the Jacobian determinant of the orthogonal projection
  from $M$ to $E$.
  This inequality is sharp, as we have equality when $\Omega$ is a
  geodesic ball in $E$.

  The Sobolev counterpart of this inequality is
  $$
  n\omega_n^\frac{1}{n}\Bigl(\int_MJ_E^\frac{1}{n-1}|u|^\frac{n}{n-1}dv_M\Bigr)^\frac{n-1}{n}
   \le \int_M|\nabla u|dv_M + n\int_M|H||u|dv_M
  $$
  for any function $u\in C_c^\infty(M)$.
\end{thm*}
We also obtain in this section a classical isoperimetric inequality (ie. of the form
$C\Vol(\Omega)^\frac{n-1}{n} \le \vol(\partial\Omega)+n\int_\Omega|H|dv_M$) with a constant
which is not sharp but improve by far the constants given in \cite{Michael-Simon}
and \cite{Hoffman-Spruck} (cf. theorem \ref{cor-isop} and the remark thereafter).

The fourth section is devoted to the study of certain warp product on which our
method still apply and gives weighted Sobolev inequalities.

%
%
%
%
\section{Preliminaries}

\subsubsection*{Isoperimetric and Sobolev inequalities}
It is a well known fact (due to Federer and Fleming, cf. for example \cite{Chavel}
for a proof) that, on Riemanniann manifolds, the isoperimetric
inequality is equivalent to the $L^1$ Sobolev inequality :
$C\Vol(\Omega)^\frac{n-1}{n} \le \vol(\partial\Omega)$ for any domain
$\Omega\subset M$ if and only if
$C(\int_M|u|^\frac{n}{n-1})^\frac{n-1}{n}\le\int_M|\nabla u|$ for any
$u\in C_c^\infty(M)$ (with the same constant in both inequalities).
This equivalence still holds true for the (weighted) isoperimetric
inequalities with the extra curvature term we are considering in this paper.

In the sequel, we shall prove the Sobolev statement of the inequalities. By density
of the smooth functions, the Sobolev inequality still holds for functions in Sobolev
spaces, and since $|\nabla u|=|\nabla|u||$ almost everywhere, it is sufficient
to consider nonnegative smooth functions.

As was observed in \cite{Cordero-Nazaret-Villani}, the $L^p$ Sobolev
inequalities on $\Rn$ can also be obtained using mass transportation method.
In fact, they obtain a nice duality principle, and if
$$
S_{n,p} = \inf\Bigl\{\frac{\|\nabla u\|_p}{\|u\|_\frac{np}{n-p}}\ \Bigl|\ u\in C_c^\infty(\Rn)\Bigr\}
$$
is the $L^p$ Sobolev constant of $\R^n$, then $S_{n,p}$ can also be
obtained as the following supremum over smooth functions
(cf. \cite{Cordero-Nazaret-Villani} theorem 2) :
$$
S_{n,p} = \frac{n(n-p)}{p(n-1)}
  \sup\left\{\frac{\int|v|^\frac{p(n-1)}{n-p}}%
  {\bigl(\int|y|^\frac{p}{p-1}|v(y)|^\frac{np}{n-p}dy\bigr)^\frac{p-1}{p}}
  \ \Bigl|\ v\in C_c^\infty(\Rn), \|v\|_\frac{np}{n-p}=1\right\}.
$$
As our method to
get the Sobolev inequalities is derived from the one used in
\cite{Cordero-Nazaret-Villani}, this caracterisation of $S_{n,p}$
will appear naturally.

\subsubsection*{Mass transportation problems}
Consider two Polish spaces $X_1$ and $X_2$, and a ``cost function''
$c:X_1\times X_2\to\R$. Given two probability measures $\mu$ and $\nu$
on $X_1$ and $X_2$ respectivally, the cost of a map $T:X_1\to X_2$ which
push forward $\mu$ on $\nu$ is $J(T)=\int_{X_1}c(x,Tx)d\mu$.
The problem of Monge consists in finding a map  whose cost is the infimum
of the costs of all maps pushing forward $\mu$ on $\nu$.

The problem of Monge may have no solution, and it is usefull to consider a
relaxed form : the Monge-Kantorovich problem. We now consider tranference
plans between $\mu$ and $\nu$, that is probability measures $\rho$ on
on $X_1\times X_2$ whose marginals are $\pi^1_\#\rho=\mu$ and
$\pi^2_\#\rho=\nu$ (where $\pi^i$ is the projection on $X_i$). The cost
of a transference plan $\rho$ is $J(\rho)=\int_{X_1\times X_2}c(x_1,x_2)d\rho(x_1,x_2)$,
and an optimal transference plan (ie. a solution of Monge-Kantorovich problem)
is a transference plan whose cost is the infimum of the costs of all transference plan between
$\mu$ and $\nu$.

In particular, if a map $T:X_1\to X_2$
push forward $\mu$ on $\nu$, then it gives rise to a transference plan
$\rho=(Id\times T)_\#\mu$ whose support in $X_1\times X_2$ is
$\supp(\rho)=\{(x,Tx)\ |\ x\in\supp(\mu)\}$ ; if an optimal transference plan
is of this form, then the map $T$ is a solution of the problem of Monge.

The properties of optimal maps and transference plans depends on the properties
of the Polish spaces $X_1$ and $X_2$ and on the cost functions ; the main
reference on this subject is \cite{Villani1}. In the sequel we shall work with the
``quadratic cost'' : $X=Y$ and $c(x,y)=d(x,y)^2$ where $d$ is the distance on $X$.
The main result we shall use on optimal transportation is the following theorem
due to Y.~Brenier (cf. \cite{Villani1} for a proof) :
\begin{thm}
  If $\mu$ and $\nu$ are probability measures on $\Rn$ which do not
  charge small sets (ie. sets with Hausdorff dimension less than or equal
  to $n-1$), then there exist a unique optimal transference plan $\rho$ between
  $\mu$ and $\nu$.

  Moreover, $\rho=(Id\times T)_\#\mu$, where $T:\supp(\mu)\to\supp(\nu)$ is
  the gradient of a convex function.
  \label{thm-BrenieMcCann}
\end{thm}

The optimality of a transference plan is related to the c-cyclical monotonicity
of its support (cf. \cite{Villani1}). It is not true in general that a transference plan is optimal
if and only if its support is c-cyclically monotone, but in our setting, as the
cost function is continuous, we have the following criterion (cf. \cite{Pratelli}
theorem~B):
\begin{thm}
  A transference plan $\rho\in P(X\times Y)$ is optimal if and only if for
  all finite family $(x_1,y_1),\dots,(x_n,y_n)$ of points of $\supp(\rho)$ and
  for any permutation $s\in\cS_n$ we have
  $$
  \sum_{i=1}^nd^2(x_i,y_i) \le \sum_{i=1}^nd^2(x_i,y_{s(i)}).
  $$
  \label{thm-criterion}
\end{thm}
For more results on the relations between optimality of transference plans
and c-cyclical monotonicity of their supports, cf. \cite{Pratelli}.

\subsubsection*{Restriction of convex functions to submanifolds}
Considering an isometric immersion $i:M^n\to N^{n+k}$, we shall note $\cA_x$ its
second fundamental form at $x$, dans $H_x=\frac{1}{n}\sum\cA_x(e_i,e_i)$ its
mean curvature vector, where the sum is taken over an orthonormal basis
of $T_xM$.

In the sequel we shall note $\nabla$ and $\hess$ (resp. $\bnabla$ and $\bhess$)
the gradient and the Hessian on $M$ (resp. on $N$).

In particular, the second fundamental form appears when writting the Hessian
of the restriction of a function to the submanifold in term of the Hessian
of the function on the ambiant manifold. Let $F:N\to \R$ be a smooth
function and let $f=F_{|_M}$ be its restriction to $M$.
For all $x\in M$ and all $\xi,\eta\in T_xM$ we have
$$
\hess f(x)(\xi,\eta) = \bhess F(x)(\xi,\eta)
  + \langle(\bnabla F)_x,\cA_x(\xi,\eta)\rangle.
$$
As a consequence, we get the Laplacian of $f$ :
\begin{equation}
  \Delta f(x) = \tr(\bhess F(x)_{|_{T_xM}}) + n\langle(\bnabla F)_x,H_x\rangle.
  \label{eqn-laplacian_submfd}
\end{equation}

The solution of the problem of Monge is given by the gradient of a convex function,
however, there is no reason for this function to be smooth ; so we have to get
a formula similar to equation \ref{eqn-laplacian_submfd} for the Laplacian
in the sense of distribution.

Let $\bV:\Rnk\to\R$ be a convex function. It is well known that $\bV$ is locally
Lipschitz, and therefore differentiable almost everywhere. Moreover, its Hessian in
the sense of distribution is a Radon measure, and, almost everywhere,
$\bV$ has second derivative given by the absolutely continuous part of this
measure with respect to Lebesgue measure (cf. for example \cite{Evans-Gariepy}).
This second derivative is known as the Hessian in the sense of Aleksandrov,
and will be noted $\bhess_A\bV$ in the sequel.

Considering an isometric immersion $i:M^n\to\Rnk$ and a convex function $\bV:\Rnk\to\R$,
we shall prove that equation \ref{eqn-laplacian_submfd} holds ``in Aleksandrov sense''.
In fact, we only need to consider the following particular case : let $E\subset\Rnk$ be
a $n$-dimensional linear subspace, let $p$ be the orthogonal projection on $E$, let
$V:E\to\R$ be a convex function, and let $\bV=V\circ p$ ; the function $\bV$ is
convex and invariant in the directions of $E^\bot$.
In this context, we have the following proposition :

\begin{prop}
	Let $V$ and $\bV$ be as above, and suppose that $|\nabla V|\le C$
	on $E$. For any bounded domain $\Omega\subset M$, the restriction
	$\Vo:\Omega\to\R$ of $\bV$ to $\Omega$ has the following properties :
	\begin{enumerate}
		\item $\Vo$ is Lipschitz and $|\nabla\Vo|\le C$ ;
		\item there exists $h\in L^2(\Omega)$ and a nonnegative
			Radon measure $\nu$ such that, in the sense of distribution,
			$\Delta_\Dp\Vo = \nu + h$
			where $h$ and $\nu$ have the following properties :
			\begin{itemize}
				\item for all $\varphi\in C_c^\infty(\Omega)$,
					$|\int_\Omega\varphi h|\le nC\int_\Omega|\varphi||H|$ ;
				\item if $D\subset \Omega$ is a domain such that the
					orthogonal projection $p:D\to E$ is a local diffeomorphism,
					then $h=n\langle H,\bnabla\bV \rangle$
					a.e. in $D$ and the Lebesgue decomposition of
					$\nu$ reads $\nu=gdv_M+\nu_s$ with
					$g(x)=\tr(\bhess_A\bV(x)_{|_{T_xM}})$ for a.a. $x\in D$,
					and $\nu_s$ singular with respect to $dv_M$.
			\end{itemize}
	\end{enumerate}
	\label{prop-cvxe_submfd}
\end{prop}
\begin{proof}
	As $|\nabla V|\le C$, the function $\bV$ is $C$-Lipschitz and for any
	$x$, $y$ in $\Omega$ we have $|\Vo(x)-\Vo(y)|\le C|x-y|\le Cd_M(x,y)$, where $d_M$
	is the distance in $M$. Therefore, $\Vo$ is $C$-Lipschitz on $\Omega$ and, by
	Rademacher's theorem, differentiable almost everywhere with $|\nabla\Vo|\le C$.

	To prove $ii.$ we follow \cite{Evans-Gariepy}. Let
	$V_\varepsilon=\rho_\varepsilon*V$, where $\rho_\varepsilon$ is a mollifier on $E$ ;
	$\bV_\varepsilon=V_\varepsilon\circ p$ is a smooth convex function on $\Rnk$, and we note
	$\Voe$ its restriction to $\Omega$. Moreover, we have
	$\nabla V_\varepsilon=\rho_\varepsilon*\nabla V$ on $E$, and
	$|\bnabla\bV_\varepsilon|\le C$.

	By formula \ref{eqn-laplacian_submfd} and integration by part on $\Omega$
	we have
	\begin{equation}
		\int_\Omega V_{\Omega,\varepsilon}\Delta\varphi
			- n\int_\Omega\varphi\langle H,\bnabla\bV_\varepsilon\rangle
			= \int_\Omega\varphi\tr(\bhess\bV_\varepsilon\,_{|_{TM}})
		\label{eqn-laplacian_mollified}
	\end{equation}
	for all $\varphi\in C_c^\infty(\Omega)$. As $|\bnabla\bV_\varepsilon|\le C$,
	the functions $\langle H,\bnabla\bV_\varepsilon\rangle$ are uniformly bounded
	in $L^2(\Omega)$ and, by weak compacity, there exists $h\in L^2(\Omega)$
	and a sequence $\varepsilon_j\to 0$ such that
	$n\int_M\varphi\langle H,\bnabla\bV_{\varepsilon_j}\rangle\to\int_M\varphi h$
	for all $\varphi\in C_c^\infty(\Omega)$. Moreover, as
	$n|\int_M\varphi\langle H,\bnabla\bV_{\varepsilon_j}\rangle|
	\le nC\int_\Omega|\varphi||H|$ for all $j$, we also have
	$|\int_\Omega\varphi h|\le nC\int_\Omega|\varphi||H|$.

	Since $\bV_\varepsilon$ is convex, passing to the limit in
	equation \ref{eqn-laplacian_mollified} gives
	$$
	\int_\Omega V_\Omega\Delta\varphi - \int_\Omega\varphi h \ge 0,
	$$
	and by Riesz representation theorem, there exist a nonnegative Radon measure $\nu$
	on $\Omega$ such that, for all $\varphi\in C_c^\infty(\Omega)$,
	$$
	\int_\Omega V_\Omega\Delta\varphi - \int_\Omega\varphi h = \int_\Omega\varphi d\nu,
	$$
	which implies that, in the sense of distribution, $\Delta_\Dp V_\Omega = \nu + h$.

	Let $D\subset\Omega$ be a domain such that $p:D\to E$ is a local diffeomorphism~;
	in particular a.a. points of $D$ are Lebesgues points of
	$\bnabla\bV$ and $\bV$ is twice differentiable a.e. in $D$. We have that
	$\bnabla\bV_{\varepsilon_j}\to\bnabla\bV$ a.e. in $D$, and, by the dominated
	convergence theorem, $\int_D\varphi\langle H,\bnabla\bV_{\varepsilon_j}\rangle
	\to \int_D\varphi\langle H,\bnabla\bV\rangle$ for all
	$\varphi\in C_c^\infty(D)$ ; this implies that $h=n\langle H,\bnabla\bV\rangle$
	a.e. in $D$.

	As the last point we want to prove is of local nature,
	we can assume that $p:D\to E$ is
	a diffeomorphism. For any $z\in E^\bot$, let $D_z=\{y+z\ |\ y\in D\}$,
	and note $V_D$ and $V_{D_z}$ the restrictions of $\bV$ to $D$ and $D_z$
	respectivelly. The set $\bD=\{y+z\ |\ y\in D,z\in E^\bot\}$ is open in $\Rnk$.
	Considering the diffeomorphism $\Phi:D\times E^\bot\to \bD$ defined by
	$\Phi(y,z)=y+z$, we can write the Lebesgue measure $\lambda_{n+k}$ on $\bD$
	in term of the Riemannian measure $dv_M$ on $D$ and the Lebesgue measure
	$\lambda_k$ on $E^\bot$ : $\lambda_{n+k}=J(y)dv_M\lambda_k$, where $J$ is
	the absolute value of the Jacobian determinant of $p$ (in particular,
	$J$ is smooth and positive). For any function $F$ on $\bD$ we have
	\begin{equation}
		\int_{\bD} F(x)dx = \int_D\int_{E^\bot}F(y+z)J(y)dv_M(y)dz.
		\label{eqn-integration-y-z}
	\end{equation}

	Considering now the smooth functions $\bVe$, we note $V_{D,\varepsilon}$ its
	restriction to $D$. Using that $\bVe$ is invariant in the directions of $E^\bot$
	we have, for any function $\bphi\in C_c^\infty(\bD)$,
	\begin{eqnarray*}
		\int_{\bD}\tr(\bhess\bVe\,_{|_{T_xD_z}})\frac{\bphi(x)}{J(y)}dx
			& = & \int_D\int_{E^\bot}\tr(\bhess\bVe\,_{|_{T_{y+z}D_z}})\bphi(y+z)dzdv_M(y) \\
		 & = & \int_D\tr(\bhess\bVe\,_{|_{T_{y}D}})\int_{E^\bot}\bphi(y+z)dzdv_M(y) \\
		 & = & \int_D\bigl(\Delta V_{D,\varepsilon}-n\langle H,\bnabla\bVe \rangle\bigr)
		 	\Bigl(\int_{E^\bot}\bphi(y+z)dz\Bigr)dv_M(y),
	\end{eqnarray*}
	where, for $x\in\bD$, $y$ and $z$ are the points in $D$ and $E^\bot$ defined
	by $x=y+z$.
	
	Let $\varphi\in C_c^\infty(D)$, let $\rho\in C_c^\infty(E^\bot)$ be such that
	$\int_{E^\bot}\rho=1$, and let $\bphi$ be defined by $\bphi(y+z)=\varphi(y)\rho(z)$.
	We get
	\begin{eqnarray*}
		\int_{\bD}\tr(\bhess\bVe\,_{|_{T_xD_z}})\frac{\bphi(x)}{J(y)}dx
		 & = & \int_D\bigl(\Delta V_{D,\varepsilon}
		 	- n\langle H,\bnabla\bVe \rangle\bigr)\varphi dv_M \\
		 & = &  \int_D\bigl(V_{D,\varepsilon}\Delta\varphi
		 	- n\varphi\langle H,\bnabla\bVe \rangle\bigr)dv_M,
	\end{eqnarray*}
	and letting $\varepsilon$ tend to $0$ gives
	\begin{equation}
		\int_{\bD}\tr(\bhess_{\Dp}\bV\,_{|_{T_xD_z}})\frac{\bphi(x)}{J(y)}dx =
		\int_D\bigl(V_D\Delta\varphi -n\varphi\langle H,\bnabla\bV \rangle\bigr)dv_M.
		\label{eqn-Hess-distrib}
	\end{equation}
	As $\bV$ is a convex function on $\Rnk$, $\tr(\bhess_{\Dp}\bV\,_{|_{T_xD_z}})$ is
	a Radon measure of the form
	$$
	\tr(\bhess_{\Dp}\bV\,_{|_{T_xD_z}}) = \tr(\bhess_A\bV\,_{|_{T_xD_z}})\lambda_{n+k}
		+ \bar{\mu}_s
	$$
	with $\bar{\mu}_s$ a singular measure. Moreover, the invariance of $\bV$ in the
	directions of $E^\bot$ implies that $\tr(\bhess_A\bV\,_{|_{T_xD_z}})$ is also
	invariant, and $\bar{\mu}_s=\mu_s\otimes\lambda_k$ with $\mu_s$ a singular measure
	on $E$. Finally, using \ref{eqn-integration-y-z}, equality \ref{eqn-Hess-distrib} becomes
	\begin{eqnarray*}
		\int_D\bigl(V_D\Delta\varphi -n\varphi\langle H,\bnabla\bV \rangle\bigr)dv_M & = &
			\int_D\tr((\bhess_A\bV)_{|_{T_yD}})\varphi(y)dv_M(y) \\
			& & + \int_{p(D)}\frac{\varphi(p^{-1}(u))}{J(p^{-1}(u))}d\mu_s(u) \\
		& = & \int_D\tr((\bhess_A\bV)_{|_{T_yD}})\varphi(y)dv_M(y) \\
			& & + \int_D\frac{\varphi}{J}d(p^{-1})_\#\mu_s,
	\end{eqnarray*}
	and we get
	$$
	\Delta_{\Dp}V_D = \bigl(\tr(\bhess_A\bV\,_{|_{TD}}) + n\varphi\langle H,\bnabla\bV \rangle\bigr)dv_M
		+ \frac{1}{J}(p^{-1})_\#\mu_s.
	$$
\end{proof}

\begin{rem}
	Denote by $V_M$ the restriction of $\bV$ to $M$.
	As a consequence of the above proposition, we have that the Laplacian
	of $V_M$ in the sense of distributions is a Radon measure ; in the sequel
	we shall note $\Delta_AV_M$ the density of its regular part in the
	Lebesgue decomposition with respect to $dv_M$.

	In particular, if $D\subset M$ is a bounded domain such that $p:D\to E$
	is a local diffeomorphism, then $\bnabla\bV$ is well defined a.e. on $D$
	and we have
	$$
	\Delta_AV_M = \tr(\bhess_A\bV\,_{|_{TM}}) + n\langle H,\bnabla\bV \rangle.
	$$
	This has to be seen has the generalisation of formula \ref{eqn-laplacian_submfd}
	to nonsmooth convex functions which are invariant in the directions of $E^\bot$.
	\label{rem-laplacian_cvx_submfd}
\end{rem}
%
%
%
%
\section{Optimal transportation and orthogonal projection on a subspace}

\subsubsection*{The general case}

As a direct consequence of theorem \ref{thm-criterion}, we have that projections (if well
defined) are optimal transportations. Consider a Polish space $X$ and  a closed
subset $C\subset X$ on which the projection $p:X\to C$ is well defined~: for
all $x\in X$ the function $d(x,.):C\to\R$ admits a unique minimum, $p(x)$ being,
by definition, the point where this minimum is achieved. For any measure
$\mu\in P(X)$, $\rho = (Id\times p)_\#\mu$ is a transference plan between the
measures $\mu$ and $\nu=p_\#\mu$. Applying theorem \ref{thm-criterion}, it is easy to
see that this transference plan is optimal : consider $(x_1,y_1),\dots,(x_n,y_n)$
in the support of $\rho$, for all $1\le i\le n$ we have $y_i=p(x_i)$ so that
for any permutation $s\in\cS_n$ and any $1\le i\le n$ we get
$d^2(x_i,y_i) \le d^2(x_i,y_{s(i)})$, which implies that $\rho$ is optimal.
A particular case is when $C$ is a linear subspace of $\R^n$, $p$ being the
orthogonal projection on $C$.

In the sequel we consider the product of three Polish spaces $X_i$, $i=1,2,3$,
and we note $\pi^{ij}:X_1\times X_2\times X_3\to X_i\times X_j$ the projection
(ie. $\pi^{ij}(x_1,x_2,x_3)=(x_i,x_j)$).
\begin{defi}{(Gluing of transference plans)}
  Consider three measures $\mu_i\in P(X_i)$, $i=1,2,3$, and two transference
  plans $\rho_{12}\in P(X_1\times X_2)$
  between $\mu_1$ and $\mu_2$, and $\rho_{23}\in P(X_2\times X_3)$ between
  $\mu_2$ and $\mu_3$.

  A gluing of $\rho_{12}$ and $\rho_{23}$ is a probability measure
  $\Gamma\in P(X_1\times X_2\times X_3)$ whose marginals on $X_1\times X_2$
  and $X_2\times X_3$ are $\rho_{12}$ and $\rho_{23}$ respectivally.
  \label{defi-gluing}
\end{defi}

As soon as the second marginal of the first transference plan equals the first
marginal of the second one, gluing of transference plans always exist (cf.
the ``gluing lemma'' in \cite{Villani2}), and they can be seen as a way of composing
transference plans : with the notation of definition \ref{defi-gluing},
we have that $\pi^{13}_\#\Gamma$ is a transference plan between $\mu_1$ and
$\mu_3$.

This is well illustrated by the particular case where $\mu_2=F^1_\#\mu_1$
and $\mu_3=F^2_\#\mu_2$. Consider the transference plans
$\rho_{12}=(Id\times F^1)_\#\mu_1$ and $\rho_{23}=(Id\times F^2)_\#\mu_2$.
Suppose $\Gamma$ is a gluing of $\rho_{12}$ and $\rho_{23}$. For any
$(x_1,x_2,x_3)\in\supp(\Gamma)$, we have $(x_1,x_2)\in\supp(\rho_{12})$ and
$(x_2,x_3)\in\supp(\rho_{23})$, so we get $x_2=F^1(x_1)$ and $x_3=F^2(x_2)$.
From this we can conclude that $\Gamma=(Id\times F^1\times F^2\circ F^1)_\#\mu_1$,
and that $\pi^{13}_\#\Gamma=(Id\times F^2\circ F^1)_\#\mu_1$ which is the transference
plan associated to the map $F^2\circ F^1$ : the gluing of transference plans
extends the composition of maps.

In general, there is no reason for $\pi^{13}_\#\Gamma$ to be optimal, even if
$\rho_{12}$ and $\rho_{23}$ are optimal, however, in the setting of projections
on a linear subspace, we have the following result :

\begin{thm}
  Let $E$ be a linear subspace of $\R^n$, and let $p_E$ denote the orthogonal
  projection on $E$. Consider two probability measures $\mu\in P(\R^n)$ and
  $\nu\in P(E)$, the optimal transference plans $\rho = (Id\times p_E)_\#\mu$
  between $\mu$ and $(p_E)_\#\mu$, and an optimal tansference plan $\sigma$
  between $(p_E)_\#\mu$ and $\nu$.

  If $\supp((p_E)_\#\mu)$ is compact, then, for any gluing $\Gamma$ of $\rho$ and $\sigma$,
  $\pi^{13}_\#\Gamma$ is an optimal transferance plan between $\mu$ and $\nu$.
  \label{thm-optimal_gluing}
\end{thm}
For the proof we shall use the following lemma :
\begin{lem}
  Let $X$ and $Y$ be Polish spaces and let $\rho\in P(X\times Y)$ be a
  tansference plan between two measures $\mu\in P(X)$ and $\nu\in P(Y)$.

  If $\supp(\nu)$ is compact, then for all $x\in\supp(\mu)$ there
  exists $y\in\supp(\nu)$ such that $(x,y)\in\supp(\rho)$.
  \label{lem-1}
\end{lem}

\begin{proof}
  First, it is easy to see that $\supp(\rho)\subset\supp(\mu)\times\supp(\nu)$.
  Now, suppose $x\in\supp(\mu)$ ; for any $\varepsilon>0$, we have
  $0<\mu(B_x(\varepsilon))=\rho(B_x(\varepsilon)\times Y)$, and there exists
  $(x_\varepsilon,y_\varepsilon)\in\supp(\rho)\cap(B_x(\varepsilon)\times Y)$.

  In particular, we have $x_\varepsilon\in B_x(\varepsilon)$ and
  $y_\varepsilon\in\supp(\nu)$ which is compact. Therefore, there exists
  $y\in\supp(\nu)$ and a sequence $(x_k,y_k)_{k\in\N}$ of points in $\supp(\rho)$
  tending to $(x,y)$. As $\supp(\rho)$ is closed, we have $(x,y)\in\supp(\rho)$
  which concludes the proof.
\end{proof}

\begin{rem}
  The previous lemma is false without the compactness of $\supp(\nu)$.
\end{rem}

\begin{proof}[proof of theorem \ref{thm-optimal_gluing}]
  Consider $n$ points $(x_1,z_1),\dots,(x_n,z_n)$ in $\supp(\pi_\#^{13}\Gamma)$.
  By lemma \ref{lem-1}, there exists points $y_1,\dots,y_n$ in $\supp((p_E)_\#\mu)$ such that
  $(x_i,y_i,z_i)\in\supp(\Gamma)$ for all $i$. Morever, as $(x_i,y_i)\in\supp(\rho)$, we have
  that $y_i=p_E(x_i)$, $(x_i,p_E(x_i),z_i)\in\supp(\Gamma)$, and $(p_E(x_i),z_i)\in\supp(\sigma)$
  for all $i$.

  Let $s\in\cS_n$. Using Pythagora's formula we have
  $$
  \sum_1^n d^2(x_i,z_i) = \sum_1^n d^2(x_i,p_E(x_i)) + \sum_1^n d^2(p_E(x_i),z_i).
  $$
  As $\sigma$ is an optimal transference plan, using theorem \ref{thm-criterion} we
  get
  $$
  \sum_1^n d^2(x_i,z_i) \le \sum_1^n d^2(x_i,p_E(x_i)) + \sum_1^n d^2(p_E(x_i),z_{s(i)}).
  $$
  Using Pythagora's formula once again we have
  $$
  \sum_1^n d^2(x_i,z_i) \le \sum_1^n d^2(x_i,z_{s(i)}).
  $$
  This implies the optimality of $\pi_\#^{1,3}\Gamma$ by  theorem \ref{thm-criterion}.
\end{proof}

As soon as we are working with the square of the distance in the Euclidean
space, it is not surprising that Pythagora's formula naturally appears,
and it has an other consequence on the geometry of Wasserstein space :
if $E_1$ and $E_2$ are two orthogonal subspaces of $\R^n$, then for any
measure $\mu_1$ and $\mu_2$ supported in $E_1$ and $E_2$ respectivally, all
the transference plan between $\mu_1$ and $\mu_2$ are optimal.
If $\rho$ is a transference plan between between $\mu_1$ and $\mu_2$, then
$\supp(\rho)\subset E_1\times E_2$ and its cost satisfies
\begin{eqnarray*}
	J(\rho) & = & \int_{E_1\times E_2}|x_1-x_2|^2d\rho(x_1,x_2) \\
	 & = & \int_{E_1\times E_2}(|x_1|^2+|x_2|^2)d\rho(x_1,x_2) \\
	 & = & \int_{E_1}|x_1|^2d\mu_1(x_1)+\int_{E_2}|x_2|^2d\mu_2(x_2).
\end{eqnarray*}
Therefore, all the transference plan have the same cost, and they all are
optimal.

For example, let $D\subset\R^2$ be the unit disc and let $\mu$ be the normalized
lebesgue measure on $D$. Consider the two inclusions $i_k:\R^2\to\R^4\simeq\R^2\times\R^2$,
$k=1,2$ defined by $i_1(x)=(x,0)$ and $i_2(x)=(0,x)$, and the two measures
$\mu_1=(i_1)_\#\mu$ and $\mu_2=(i_2)_\#\mu$. For any $t\in\R$ the map
$F_t:\supp(\mu_1)\to\supp(\mu_2)$ defined by $F_t(x,0)=(0,\ex^{it}x)$ push
forward $\mu_1$ on $\mu_2$ and, because of the preceding remark, gives rise
to an optimal transference plan.

Now, using displacement interpolation (cf. for example \cite{Lott-Villani} \S 2), each of these
optimal transference plans gives rise to a geodesic in the Wasserstein space
$P_2(\R^4)$, and we constructed a continuous family of geodesics in $P_2(\R^4)$
with common end points and having the same length.

It is easy to find such a phenomenum in the Wasserstein space of a Riemannian
manifold with positive curvature : considering for example the Dirac masses on the
north and south pole of the sphere, each geodesic between the poles gives rise
to a geodesic in the Wasserstein space (the map $x\mapsto\delta_x$ is an isometric
embedding between the manifold and its Wasserstein space). However, our example
is of different nature as there is a unique geodesic between any two points in $\Rn$.

This situation is of ``positive curvature'' nature : on a Riemanniann manifold,
such a situation implies that the end points are conjugate
points along the geodesics, and therefore implies the presence of positive sectionnal
curvature. Therefore, although the Euclidean space has vanishing curvature, its
Wasserstein has positive curvature in some sense ; this remark has to be compared
with J. Lott's curvature calculations
on the spaces of measures with $C^\infty$ densities with respect to the
Lebesgue measure (cf. \cite{Lott}, corollary 1).

\subsubsection*{The case of measures supported in a submanifold}
In the sequel we want to use solutions to the problem of Monge to compare measures
supported in a submanifold with measures supported in a linear subspace.
By the previous theorem, it is natural to consider the push forward of the first measure
by the orthogonal projection on the linear subspace, and to use a solution
of the problem of Monge in the linear subspace.
A sufficient condition for such a solution to exist, is that the pushed measure
does not give mass to small sets of the linear subspace.

Consider an isometric immersion $i:M^n\to\Rnk$, and let $E$ be a linear subspace
of $\Rnk$. We shall note $P:\Rnk\to E$ the orthogonal projection on $E$, $p=P_{|_M}$
its restriction to $M$, and $\cC=\{x\in M\ |\ T_xp:T_xM\to E \mbox{ is not onto }\}$
the critical set of $p$. In particular, $\cC$ is a closed subset of $M$.

\begin{prop}
  Let $i:M^n\to\R^{n+k}$ be an isometric immersion, let $E$ be a linear subspace
  of $\R^{n+k}$ with $\dim(E)\le n$, and let $p:M\to E$ be the orthogonal projection on $E$.

  For any nonnegative function $f$ on $M$ vanishing on $\cC$, the measure $\mu=fdv_M$
  is such that $p_\#\mu$ is absolutly continuous with respect to the
  Lebesgue measure of $E$.
\end{prop}
\begin{proof}
  Let $A\subset E$ be a Borelian subset such that $p_\#\mu(A)>0$. As $\mu(p^{-1}(A))>0$,
  there exists $x\in p^{-1}(A)$ such that $f(x)>0$, and a neighborhood $U$ of $x$
  such that $p_{|_U}$ is a submersion and $\mu(U\cap p^{-1}(A))>0$. Since
  $p_{|_U}$ is a submersion we have $\lambda(p(U\cap p^{-1}(A)))>0$
  which implies that $\lambda(A)>0$.
\end{proof}

As a consequence, we have the following result on the existence of a solution for
the problem of Monge between $\mu$ and any measure on $E$ :
\begin{cor}
  For any nonnegative function $f$ with compact support on $M$ and vanishing
  on $\cC$, and for any measure
  $\nu$ on $E$, the problem of Monge between the measures $\mu=fdv_M$
  and $\nu$ admits a solution $T:M\to E$.

  Moreover, there exists a convex function $V$ on $\R^{n+k}$ such that
  $T$ is the restriction to $M$ of the gradient of $V$.
\end{cor}
\begin{proof}
  Using the proposition above and Brenier's Theorem, the problem of Monge between $p_\#\mu$
  and $\nu$ has a solution $S=\nabla W$ in $E$, where $W$ is a convex function on $E$.

  By theorem \ref{thm-optimal_gluing}, $T=S\circ p=\nabla W\circ p=\nabla(W\circ p)$
  is a solution to the problem of Monge between $\mu$ and $\nu$, and $V=W\circ p$ is
  the desired convex function on $\Rnk$.
\end{proof}

\begin{rem}
 Although the result above looks like Brenier's theorem, there are some differences.
 In particular, even if $\nu$ does not give mass to small sets in $E$, the problem
 of Monge between $\nu$ and $\mu$ could have no solution as the projection $p:M\to E$
 may not be one to one.
\end{rem}

Let us now consider the case where $\dim(E)=\dim(M)=n$, and assume that the measure
$\mu=fdv_M$ has compact support (with $f$ still vanishing on $\cC$).
In the sequel we shall note $J_E(x)$ the absolute value
of the Jacobian determinant of $p$ at $x$. (ie. $J_E(x)=|\det(T_xp)|$,
where the determinant is taken in orthonormal basis of $T_xM$ and $E$).

If $y\in E$ is such that $p^{-1}(y)\cap\supp(\mu)$
is not finite, then, by the compactness of $\supp(\mu)$, $y$ must be a critical value
of $p$. As a consequence of Morse-Sard's theorem (cf. for example
\cite{Hirsch}), we have that $p^{-1}(y)\cap\supp(\mu)$
is finite for almost all $y\in E$ with respect to Lebesgue measure $\lambda$.

Using this fact, we have $p_\#\mu = F\lambda$ where
\begin{equation}
  F(y) = \sum_{x\in p^{-1}(y)\cap\supp(\mu)}\frac{f(x)}{J_E(x)}
  \label{eqn-density}
\end{equation}
is well defined for almost all $y\in E$.

In the sequel we shall need a regularity result for the solution of the problem
of Monge ; it is given by the following proposition :
\begin{prop}
  If $f\in C_c^\infty(M\setminus\cC)$ and $g\in C^\infty(\overline{D})$ where $D$
  is a smooth convex domain in $E$,
  then there exists a smooth convex function $W$ on $E$ such that $\nabla(W\circ p)$
  is a solution to the problem of Monge between $\mu=fdv_M$ and $\nu=g\lambda$.
  \label{prop-regularity}
\end{prop}
\begin{proof}
  The smoothness of $W$ will be a consequence of Caffarelli's
  regularity theory for solutions of the problem of Monge (cf. \cite{Caffarelli1},
  \cite{Caffarelli2} and \cite{Caffarelli3}). In order
  to use this theory, we just have to prove that the density $F$ of
  $p_\#\mu$ with respect to Lebesgue measure belongs to $C_c^\infty(E)$.

  As $\supp(\mu)$ is compact, so is $\supp(p_\#\mu)$. Let $y\in\supp(p_\#\mu)$,
  $p^{-1}(y)\cap\supp(\mu)$ is finite, and for each $x\in p^{-1}(y)\cap\supp(\mu)$
  there exists a neighborhood $U_x$ of $x$ such that $p:U_x\to p(U_x)$ is a
  diffeomorphism. Moreover we can assume that for all $x$, $p(U_x)=B_\varepsilon(y)$.

  Since $\supp(\mu)\setminus\cup_xU_x$ is compact in $\Rnk$, there exist
  $0<\alpha\le\varepsilon$ such that the cylinder $B_\alpha(y)+E^\bot$ does not
  intersect $\supp(\mu)\setminus\cup_xU_x$. Therefore, on $B_\alpha(y)$, $F$ is
  a sum of smooth functions, and $F$ is smooth on $E$.
\end{proof}

%
%
%
%
\section{Isoperimetric inequalities for submanifolds of the Euclidean space}

In this section we consider an isometric immersion $i:M^n\to\Rnk$,
and a linear subspace $E\subset\Rnk$ of dimension $n$.

For any $n$-plane $F\subset\Rnk$, let $K_E(F)=|\det(q)|$ where
$q:F\to E$ is the orthogonal projection from $F$ to $E$ and $\det(q)$
is taken in orthonormal basis of $F$ and $E$.

In particular, if $p:M\to E$ denote the orthogonal projection on $E$, and
$J_E(x)=|\det(T_xp)|$, we have $J_E(x)=K_E(T_xM)$.

\subsubsection*{A weighted isoperimetric inequality}
\begin{thm}
  Let $i:M^n\to\Rnk$ be an isometric immersion, and let $E$ be a $n$-dimensional
  linear subspace of $\Rnk$. For any regular domain $\Omega\subset M$ we have
  $$
  n\omega_n^\frac{1}{n}\Bigl(\int_\Omega J_E^\frac{1}{n-1}dv_M\Bigr)^\frac{n-1}{n}
   \le \vol(\partial\Omega) + n\int_\Omega|H|dv_M.
  $$

  The Sobolev counterpart of this inequality is
  $$
  n\omega_n^\frac{1}{n}\Bigl(\int_MJ_E^\frac{1}{n-1}|u|^\frac{n}{n-1}dv_M\Bigr)^\frac{n-1}{n}
   \le \int_M|\nabla u|dv_M + n\int_M|H||u|dv_M
  $$
  for any function $u\in C_c^\infty(M)$.

  These inequalities are sharp.
  \label{thm-sobL1}
\end{thm}
\begin{proof}
  Let $u\in C_c^\infty(M)$ be a nonnegative function and let
  $f=\frac{J_E^\frac{1}{n-1}u^\frac{n}{n-1}}{c_E(u)}$,
  where $c_E(u)=\int_MJ_E^\frac{1}{n-1}u^\frac{n}{n-1}dv_M$.
  The function $f$ vanishes on $\cC$, therefore,
  the measure $\mu=fdv_M$ is such that $p_\#\mu$
  is absolutly continuous with respect to Lebesgue measure on $E$ with a
  density $F$ given by the formula \ref{eqn-density}.

  Using Brenier's theorem, there exists a convex function $V$
  such that $\nabla V$ is the solution of the problem of Monge in $E$
  between $p_\#\mu$ and $\frac{\chi_{B_E}}{\omega_n}dz$, where $B_E$ is
  the unit ball in $E$. Moreover, by Brenier's theorem, we have that
  $\nabla V(\supp(p_\#\mu))\subset B_E$, so that $|\nabla V|\le 1$
  on $\supp(p_\#\mu)$, and we can assume that $V$ is finite on $E$. In fact,
  if this is not the case, just replace $V$ by
  $$
  W(x) = \sup\Bigl\{a(x)\ \Bigr|\ a \mbox{ affine function},\ |\nabla a|\le 1,\
  	a\le V \mbox{ on } \supp(p_\#\mu) \Bigr\}.
  $$
  This function is convex on $E$ with $|\nabla W|\le 1$, and $W=V$ on
  $\supp(p_\#\mu)$ so that $\nabla W$ push forward $p_\#\mu$ on
  $\frac{\chi_{B_E}}{\omega_n}dz$. In the sequel we shall assume that
  $V$ is finite on the whole of $E$.

  Let $\bV$ denotes the extension of $V$ to $\Rnk$ (that
  is $\bV=V\circ p$), and $V_M$ denotes the restriction of $\bV$ to $M$.
  The singular set of $\bV$ (i.e. the set where $\bV$ is not twice
  differentiable) is the preimage by $p$ of the singular set of $V$, and
  since $p$ is a local diffeomorphism on $\supp(\mu)$, $\bV$ and $V_M$
  are twice differentiable almost everywhere in $\supp(\mu)$.

  Consider now the change of variable $z=\nabla V(y)$ in $E$. As in
  \cite{Cordero-Nazaret-Villani}, using a remark due to McCann, this
  change of variable gives $\omega_nF(y)=|\det(\hess_AV(y))|$,
  and by \ref{eqn-density} we get
  $$
  \omega_n\frac{f(x)}{J_E(x)} \le \omega_nF(p(x)) = |\det(\hess_AV(p(x)))|
  $$
  for almost all $x$ in the support of $\mu$.

  From the definition of the function $\bV$, we have that its Hessian
  is given by
  $\bhess_A\bV(x)(\xi,\eta)=\hess_AV(p(x))(P(\xi),P(\eta))$ for a.a. points
  $x\in\Rnk$ and any vectors $\xi$ and $\eta$, where $P$ is the
  orthogonal projection on $E$. As the orthogonal projection on $E$
  is also the tangent map of $p$, it follows that, for a.a. $x\in\supp(\mu)$,
  $$
  \det(\bhess_A\bV(x)_{|_{T_xM}}) = J_E^2(x)\det(\hess_AV(p(x))),
  $$
  from which we deduce
  $$
  \omega_nJ_E(x)f(x) \le |\det(\bhess_A\bV(x)_{|_{T_xM}})|.
  $$
  As the restriction of a nonnegative matrix is still nonnegative, the
  arithmetic-geometric inequality gives
  \begin{equation}
    n\omega_n^\frac{1}{n}J_E(x)^\frac{1}{n}f(x)^\frac{1}{n} \le
      \tr(\bhess_A\bV(x)_{|_{T_xM}}).
    \label{eqn-ag}
  \end{equation}
  As $f$ vanishes on $\cC$, proposition \ref{prop-cvxe_submfd} and
  remark \ref{rem-laplacian_cvx_submfd} imply that a.e. in $\supp(\mu)$
  \begin{equation}
  n\omega_n^\frac{1}{n}J_E(x)^\frac{1}{n}f(x)^\frac{1}{n} \le
    \Delta_AV_M - n\langle H,\bnabla\bV\rangle,
    \label{eqn-step2}
  \end{equation}
  where $H$ is the mean curvature vector of $M$.

  Multiplication by $u$ of the previous inequality gives
  \begin{equation}
    \frac{n\omega_n^\frac{1}{n}}{c_E(u)^\frac{1}{n}}
      J_E^\frac{1}{n-1}u^\frac{n}{n-1}
      \le u\Delta_AV_M - nu\langle H,\bnabla\bV\rangle.
    \label{eqn-step3}
  \end{equation}

  By proposition \ref{prop-cvxe_submfd} we have that $\Delta_\Dp\Vo = \nu + h$
  with $\nu$ a nonnegative Radon measure. Using remark \ref{rem-laplacian_cvx_submfd}
  and the Lebesgue decomposition $\nu = \nu_{ac} + \nu_s$, we get
  $$
  \int_{M\setminus\cC}u\Delta_AV_M - nu\langle H,\bnabla\bV\rangle
  	= \int_{M\setminus\cC}ud\nu_{ac},
  $$
  and since $\nu$ and $u$ are nonnegative we obtain
  \begin{eqnarray}
  	\int_{M\setminus\cC}\bigl(u\Delta_AV_M - nu\langle H,\bnabla\bV\rangle\bigr)dv_M
  		& \le & \int_{M\setminus\cC}ud\nu \nonumber \\
  	 & \le & \int_Mud\nu \nonumber \\
  	 & \le & \int_Mu\Delta_\Dp V_Mdv_M  \nonumber \\
  	 	& & - \int_Muhdv_M \label{eqn-step5} \\
  	 & \le & -\int_M\langle\nabla u,\nabla V_M\rangle dv_M \nonumber \\
  	 	& & + n\int_Mu|H|dv_M.
  	\label{eqn-step4}
  \end{eqnarray}
  As $|\bV|\le 1$, we also have $|\nabla V_M|\le 1$ on $M$, and, since
  the lefthandside of equation \ref{eqn-step3} vanishes on $\cC$,
  integrating this equation on $M\setminus\cC$ gives the desired Sobolev
  inequality :
  $$
  n\omega_n^\frac{1}{n}\Bigl(\int_MJ_E^\frac{1}{n-1}u^\frac{n}{n-1}dv_M\Bigr)^\frac{n-1}{n}
    \le \int_M|\nabla u|dv_M + n\int_Mu|H|dv_M.
  $$

  The isoperimetric companion of this Sobolev inequality is
  $$
  n\omega_n^\frac{1}{n}\Bigl(\int_\Omega J_E^\frac{1}{n-1}dv_M\Bigr)^\frac{n-1}{n}
   \le \vol(\partial\Omega) + n\int_\Omega|H|dv_M,
  $$
  and this inequality is sharp as we have equality if $M=E$ and
  $\Omega$ is a ball.
\end{proof}

\subsubsection*{The classical isoperimetric inequality}
To get the usual isoperimetric inequality (without any weight), we can
perform an integration on the Grassmannian of $n$-plane in $\Rnk$.

Let $F$ be a $n$-plane in $\Rnk$, and let
$$
\alpha_{n,k} = \frac{1}{\Vol(G_{n,n+k})}\int_{G_{n,n+k}}K_E(F)^\frac{1}{n}dE,
$$
where the integration is taken for the Haar measure of $G_{n,n+k}$.
Using the homogeneity of $G_{n,n+k}$ and the invariance of the Haar measure,
it is easy to see that
$\alpha_{n,k}$ does not depend on the choice of $F$.

\begin{thm}
  Let $i:M^n\to\Rnk$ be an isometric immersion, and let $E$ be a $n$-dimensional
  linear subspace of $\Rnk$. For any regular domain $\Omega\subset M$ we have
  $$
  n\omega_n^\frac{1}{n}\alpha_{n,k}\Vol(\Omega)^\frac{n-1}{n}
   \le \vol(\partial\Omega) + n\int_\Omega|H|dv_M.
  $$

  The Sobolev counterpart of this inequality is
  $$
  n\omega_n^\frac{1}{n}\alpha_{n,k}\Bigl(\int_M|u|^\frac{n}{n-1}dv_M\Bigr)^\frac{n-1}{n}
   \le \int_M|\nabla u|dv_M + n\int_M|H||u|dv_M
  $$
  for any function $u\in C_c^\infty(M)$.
  \label{cor-isop}
\end{thm}
\begin{proof}
  Choose $a>0$, and let $f=\frac{J_E^au^\frac{n}{n-1}}{c_{E,a}(u)}$,
  where $c_{E,a}(u)=\int_MJ_E^au^\frac{n}{n-1}$.
  Following the previous proof, equation \ref{eqn-step3} becomes
  $$
  \frac{n\omega_n^\frac{1}{n}}{c(u)^\frac{1}{n}}
    J_E^{\frac{a+1}{n}}u^\frac{n}{n-1}
    \le u\Delta_AV_M - nu\langle H,\bnabla\bV\rangle
  $$
  a.e. in $M\setminus\cC$, where we also used that
  $c_{E,a}(u)\le c(u)=\int_Mu^\frac{n}{n-1}$.
  Integrating on $M\setminus\cC$, using inequality \ref{eqn-step4}
  and letting $a\to 0$ gives
  $$
  \frac{n\omega_n^\frac{1}{n}}{c(u)^\frac{1}{n}}
    \int_MJ_E^\frac{1}{n}u^\frac{n}{n-1}dv_M
    \le \int_M|\nabla u|dv_M + n\int_Mu|H|dv_M.
  $$
  As $J_E(x)=K_E(T_xM)$, integrating on $G_{n,n+k}$ with respect to $E$ we get
  $$
  n\omega_n^\frac{1}{n}\alpha_{n,k}\Bigl(\int_Mu^\frac{n}{n-1}dv_M\Bigr)^\frac{n-1}{n}
    \le \int_M|\nabla u|dv_M + n\int_Mu|H|dv_M.
  $$
  The isoperimetric companion of this Sobolev inequality is
  $$
  n\omega_n^\frac{1}{n}\alpha_{n,k}\Vol(\Omega)^\frac{n-1}{n}
    \le \vol(\partial\Omega) + n\int_\Omega|H|dv_M
  $$
  for any regular domain $\Omega\subset M$.
\end{proof}

The isoperimetric inequality obtained in this theorem is not the expected one,
as $\alpha_{n,k}<1$. However, we have that $\lim_{n\to\infty}\alpha_{n,1}=1$,
so that this inequality is not far from being sharp for hypersurfaces
of high dimension.

To compute the limit, note that
$\alpha_{n,1}=\frac{1}{\vol(S^n)}\int_{S^n}|\langle\eta,\xi\rangle|^\frac{1}{n}dv_{S^n}(\xi)$,
for a given $\eta\in S^n$. Taking normal coordinates on $S^n$ centered at $\eta$ we get
$$
\alpha_{n,1}=\frac{\vol(S^{n-1})}{\vol(S^n)}
	\int_0^\pi|\cos r|^\frac{1}{n}\sin^{n-1}rdr
	= \frac{\int_0^\pi|\cos r|^\frac{1}{n}\sin^{n-1}rdr}{\int_0^\pi\sin^{n-1}rdr}.
$$
Using that $|\cos r|\ge \cos(\frac{\pi}{2}-\frac{1}{n})
\chi_{[0,\frac{\pi}{2}-\frac{1}{n}]\cup[\frac{\pi}{2}+\frac{1}{n},\pi]}$, we have
\begin{eqnarray*}
	\alpha_{n,1} & \ge &
		\frac{\cos^\frac{1}{n}(\frac{\pi}{2}-\frac{1}{n})
		\Bigl(\int_0^\pi\sin^{n-1}rdr
		-\int_{\frac{\pi}{2}-\frac{1}{n}}^{\frac{\pi}{2}+\frac{1}{n}}\sin^{n-1}rdr\Bigr)}%
		{\int_0^\pi\sin^{n-1}rdr} \\
	 & \ge & \frac{\cos^\frac{1}{n}(\frac{\pi}{2}-\frac{1}{n})
		\Bigl(\int_0^\pi\sin^{n-1}rdr - \frac{1}{2n}\Bigr)}{\int_0^\pi\sin^{n-1}rdr}.
\end{eqnarray*}
As Wallis' integral satisfies
$\int_0^\pi\sin^{n-1}rdr\sim_\infty\sqrt{\frac{2\pi}{n-1}}$, this lower bound tends
to $1$ when $n$ tends to infinity.

This show that our result improve the constant of this kind
of isoperimetric inequalities for submanifolds. In fact,
the constants given in \cite{Michael-Simon} and \cite{Hoffman-Spruck} are
of the form $n\omega_n^\frac{1}{n}\beta_n$ with $\beta_n$ tending to
$0$ when the dimension tends to infinity.

Using ideas of L. Simon, P. Topping obtained the inequality
$2\pi\Vol(\Omega)\le (\vol(\partial\Omega)+2\int_\Omega|H|)^2$ for any surfaces in $\R^{2+k}$
(cf. \cite{Topping}, appendix A). A simple calculation proves that this inequality is
better than the one we get by our method.
Note that for minimal surfaces in $\R^3$, A. Ros and A. Stone
obtained the inequality $2\pi\sqrt{2}\Vol(\Omega)\le\vol(\partial\Omega)^2$
(cf. \cite{Choe2} \S 10.1 for a proof).

\subsubsection*{Transference plans ``moving with the point''}
In the preceding section, we do not get the expected isoperimetric
inequality because the Jacobian of the projection on $E$, which is less than or
equal to one, naturally appear. To avoid this problem, the idea would be to use
at each point of $M$ the projection on the tangent space $T_xM$, and hence to
use a family of transportations ``moving with the point''.

To illustrate this point, let us consider the case of hypersurfaces. Let
$i:M^n\to\Rnu$ be an isometric immersion, and let $u\in C_c^\infty(M)$ be
a nonnegative function.

Choose a nondecreasing smooth function $\varphi$ on $\R_+$ such that $\varphi$
vanishes in a neighborhood of $0$, $0\le\varphi\le 1$, and $\varphi(1)=1$.

For each $\xi\in S^n$, we consider the orthogonal projection $p_\xi:M\to\xi^\bot$,
$J_\xi$ the determinant of its Jacobian, and we note
$f_\xi=\frac{\varphi(J_\xi)u^\frac{n}{n-1}}{c_\xi(u)}$,
where $c_\xi(u)=\int_M\varphi(J_\xi)u^\frac{n}{n-1}$.

Considering the optimal transportations $T_\xi:M\to \xi^\bot$ which push forward
the measure $f_\xi dv_M$ on $M$ to the normalized Lebesgue measure of the unit
ball of $\xi^\bot$, we can define the following map
$$
\Phi:\left\{\begin{array}{rcl}
          M\times S^n & \to & \Rnu \\
          (x,\xi) & \mapsto & T_\xi(x)
         \end{array}\right..
$$
Using the Gauss map $g$ of $M$, we define
$$
X:\left\{\begin{array}{rcl}
          M & \to & \Rnu \\
          x & \mapsto & \Phi(x,g(x))
         \end{array}\right..
$$
As $X_x\in g(x)^\bot$ for each $x\in M$, $X$ is just a vector field on $M$,
and the question is~: can we use this vector field as a ``Knothe map'' to prove
some Sobolev inequality on $M$ ?

For each $\xi\in S^n$, the optimal transportation $T_\xi$ is the gradient of a
convex function $\bV_\xi$ which is the extension to $\Rnu$ of a convex function
in $\xi^\bot$. By proposition \ref{prop-regularity}, the function $\bV_\xi$ is
smooth.

In the sequel we shall note $T^x_{(x,\xi)}\Phi:T_xM\to\Rnu$
(resp. $T^\xi_{(x,\xi)}\Phi:\xi^\bot\to\Rnu$) the tangent map to $\Phi$ with respect
to the first (resp. to the second) variable.

As the dérivative of the Gauss map is given by the shape operator,
for a vector $e\in T_xM$ we have, for any $x\in M$,
$$
(e.X)(x) = (e.\bnabla\bV_\xi)_{|_{\xi=g(x)}} - T^\xi_{(x,g(x))}\Phi.S_x(e)
$$
where $S_x$ is the shape operator of $M$ at $x$. And making the sum
over an orthonormal basis of $T_xM$ we get
\begin{equation}
  \div(X)(x) = \tr(\bhess\bV_\xi(x)_{|_{T_xM}})_{|_{\xi=g(x)}}
    - \tr(T^\xi_{(x,g(x))}\Phi\circ S_x).
  \label{eqn-divX}
\end{equation}

From this expression for $\div(X)$ we can deduce the following proposition :
\begin{prop}
  Let $i:M^n\to\Rnu$ be an isometric immersion. For any regular domain
  $\Omega\subset M$ we have
  $$
  n\omega_n^{\frac{1}{n}}\Vol(\Omega)^{1-\frac{1}{n}} \le \vol(\partial\Omega)
    + \int_\Omega|\tr(T^\xi_{(x,g(x))}\Phi\circ S_x)|dv_M(x).
  $$

  The Sobolev counterpart of this inequality is
  $$
  n\omega_n^\frac{1}{n}\Bigl(\int_M|u|^\frac{n}{n-1}dv_M\Bigr)^\frac{n-1}{n}
   \le \int_M|\nabla u|dv_M + \int_M|u|\bigl|\tr(T^\xi_{(x,g(x))}\Phi\circ S_x)\bigr|dv_M(x)
  $$
  for any function $u\in C_c^\infty(M)$.

  These inequalities are sharp.
\end{prop}
\begin{proof}
  Following the proof of theorem \ref{thm-sobL1}, for any $\xi\in S^n$ and any
  $x\in\supp(f_\xi)$, equation \ref{eqn-ag} gives
  $$
  n\omega_n^\frac{1}{n}J_\xi(x)^\frac{1}{n}f_\xi(x)^\frac{1}{n} \le
      \tr(\bhess\bV_\xi(x)_{|_{T_xM}}),
  $$
  with the usual Hessian, $\bV_\xi$ being smooth.
  Using the fact that $J_{g(x)}(x)=1$ and $c_\xi(u)\le\int_Mu^\frac{n}{n-1}=c(u)$
  we get
  \begin{eqnarray*}
    n\omega_n^\frac{1}{n}\frac{u(x)^\frac{1}{n-1}}{c(u)^\frac{1}{n}} & \le &
      \tr(\bhess\bV_\xi(x)_{|_{T_xM}})_{|_{\xi=g(x)}} \\
     & \le & \div(X)(x) + \tr(T^\xi_{(x,g(x))}\Phi\circ S_x).
  \end{eqnarray*}
  Multipying by $u$, integrating by part, and using that $|X_x|\le 1$ for any
  $x\in M$ we obtain
  $$
  n\omega_n^\frac{1}{n}\Bigl(\int_Mu^\frac{n}{n-1}dv_M\Bigr)^\frac{n-1}{n}
   \le \int_M|\nabla u|dv_M + \int_Mu\tr(T^\xi_{(x,g(x))}\Phi\circ S_x)dv_M.
  $$
  The isoperimetric counterpart of this Sobolev inequality is
  $$
  n\omega_n^{\frac{1}{n}}\Vol(\Omega)^{1-\frac{1}{n}} \le \vol(\partial\Omega)
    + \int_M\tr(T^\xi_{(x,g(x))}\Phi\circ S_x)dv_M,
  $$
  and this inequality is sharp as we have equality for any geodesic ball
  lying in any hyperplane of $\Rnu$.
\end{proof}

Note that the result of the previous proposition is not so far from that
of theorem \ref{thm-sobL1}, as the third term involves the shape operator
whose trace is the mean curvature. The remaining problem is to deal with the
derivative of the transports map with respect to the parameter $\xi$.

\subsubsection*{A weighted $L^p$-sobolev inequality}
In \cite{Cordero-Nazaret-Villani} the authors also obtained the sharp
$L^p$ sobolev inequalities on $\R^n$ in a similar way, using a different target
measure (cf. \cite{Cordero-Nazaret-Villani} theorem 2). In our setting,
we get weighted Sobolev inequalities, with weights involving a negative
power of $J_E$. For this weight to be finite almost eveywhere, we shall
assume that the critical set $\cC$ of the projection is negligible in $M$.

\begin{thm}
  Let $i:M^n\to\Rnk$ be an isometric immersion, and let $E$ be a $n$-dimensional
  linear subspace of $\Rnk$ such that the critical set of the orthogonal
  projection from $M$ to $E$ is negligible.

  For any $1<p<n$, and for any function $u\in C_c^\infty(M)$ we have
  \begin{multline*}
  S_{n,p}\Bigl(\int_MJ_E^\frac{1}{n-1}|u|^\frac{np}{n-p}dv_M\Bigr)^\frac{n-p}{np}
   \le \int_MJ_E^{-\frac{p-1}{n-1}}|\nabla u|^pdv_M \\
   + \frac{n(n-p)}{p(n-1)}\int_MJ_E^{-\frac{p-1}{n-1}}|H||u|dv_M
  \end{multline*}
  where $S_{n,p}$ is the $L^p$ Sobolev constant of $\Rn$.
  This inequality is sharp.
  \label{thm-SobLp}
\end{thm}
\begin{proof}
	Let $f=\frac{J_E^\frac{1}{n-1}u^\frac{np}{n-p}}{c_E(u)}$,
	where $c_E(u)=\int_MJ_E^\frac{1}{n-1}u^\frac{np}{n-p}$.
	As $f$ vanishes on $\cC$, we have $p_\#\mu=F(y)dy$ with $F$ given
	by formula \ref{eqn-density}. We follow the proof of
	theorem \ref{thm-sobL1}, except that
	$\nabla V$ is the solution of the problem of Monge between
	the measures $p_\#\mu$ and $G(z)dz$, where the function
	$G\in C_c^\infty(E)$ will be made precise later.

	Using the change of variable formula between $F(y)dy$ and $G(z)dz$,
	the relation between $f$ and $G$ becomes
	\begin{eqnarray*}
		\frac{f(x)}{J_E(x)} & \le & F(p(x)) \\
		 & \le & G(\nabla V(p(x)))|\det(\hess_AV(p(x)))| \\
    	 & \le & G(\bnabla \bV(x))|\det(\hess_AV(p(x)))|.
	\end{eqnarray*}

	From this point, we follow the steps of the proof of theorem \ref{thm-sobL1} :
	by the arithmetic-geometric inequality and proposition \ref{prop-cvxe_submfd},
	equation \ref{eqn-step2} becomes
	$$
	J_E(x)^\frac{1}{n}G(\nabla \bV(x))^{-\frac{1}{n}} \le
    \frac{1}{n}f(x)^{-\frac{1}{n}}\Delta_AV_M - f(x)^{-\frac{1}{n}}\langle H,\bnabla\bV\rangle.
	$$
	As $\cC$ is negligible, this inequality occurs a.e. in the support of $u$.
  Multiplying both parts by $J_E(x)^{-\frac{1}{n}}f(x)$ and integrating on $M$
  we get
  \begin{multline*}
  \int_MG(\nabla \bV(x))^{-\frac{1}{n}}f(x)dv_M \le
    \frac{1}{nc_E(u)^\frac{n-1}{n}}\int_Mu^\frac{p(n-1)}{n-p}\Delta_A V_M \\
    - \frac{1}{c_E(u)^\frac{n-1}{n}}\int_Mu^\frac{p(n-1)}{n-p}\langle H,\bnabla\bV\rangle.
  \end{multline*}
  As the map $\bnabla\bV:M\to E$ push the measure $\mu=fdv_M$ on $G(y)dy$,
  the left handside of this inequality reads $\int_EG(y)^\frac{n-1}{n}dy$.
  On the right handside, we use proposition \ref{prop-cvxe_submfd} to compare
  $\Delta_AV_M$ and $\Delta_\Dp V_M$, as in equation \ref{eqn-step5} where since $\cC$ is
  negligible, $h=n\langle H,\bnabla\bV \rangle$ a.e.. Then, integration by part gives
  \begin{multline*}
  \int_EG(y)^\frac{n-1}{n}dy \le
    -\frac{p(n-1)}{n(n-p)c_E(u)^\frac{n-1}{n}}
    	\int_Mu^\frac{n(p-1)}{n-p}\langle \nabla u,\nabla V_M\rangle \\
    - \frac{1}{c_E(u)^\frac{n-1}{n}}\int_Mu^\frac{p(n-1)}{n-p}\langle H,\bnabla\bV\rangle.
  \end{multline*}
  If $q=\frac{p}{p-1}$ is the dual exponent to $p$, using H\"older inequality
  and $|\nabla V_M|\le|\bnabla\bV|$ we get
  \begin{multline*}
    \frac{n(n-p)}{p(n-1)}\int_EG(y)^\frac{n-1}{n}dy \le \\
      \frac{1}{c_E(u)^\frac{n-1}{n}}
      \Bigr(\int_MJ_E^\frac{1}{n-1}u^\frac{np}{n-p}|\bnabla\bV|^q\Bigl)^{\frac{1}{q}}
      \Bigr(\int_MJ_E^{-\frac{p-1}{n-1}}|\nabla u|^p\Bigl)^\frac{1}{p} \\
      + \frac{n(n-p)}{p(n-1)c_E(u)^\frac{n-1}{n}}
      \Bigr(\int_MJ_E^\frac{1}{n-1}u^\frac{np}{n-p}|\bnabla\bV|^q\Bigl)^{\frac{1}{q}}
      \Bigr(\int_MJ_E^{-\frac{p-1}{n-1}}|u|^p|H|^p\Bigl)^\frac{1}{p},
  \end{multline*}
  which gives
  \begin{multline*}
    \frac{n(n-p)}{p(n-1)}\int_EG(y)^\frac{n-1}{n}dy \le
      \frac{1}{c_E(u)^\frac{n-p}{np}}
      \Bigr(\int_Mf|\bnabla\bV|^q\Bigl)^{\frac{1}{q}}
      \Bigr(\int_MJ_E^{-\frac{p-1}{n-1}}|\nabla u|^p\Bigl)^\frac{1}{p} \\
      + \frac{n(n-p)}{p(n-1)c_E(u)^\frac{n-p}{np}}
      \Bigr(\int_Mf|\bnabla\bV|^q\Bigl)^{\frac{1}{q}}
      \Bigr(\int_MJ_E^{-\frac{p-1}{n-1}}|u|^p|H|^p\Bigl)^\frac{1}{p}.
  \end{multline*}
  Using once again that $\bnabla\bV$ push $\mu$ on $G(z)dz$ we obtain
  \begin{multline*}
    \frac{n(n-p)}{p(n-1)}
    \frac{\int_EG(y)^\frac{n-1}{n}dy}{\Bigr(\int_E|y|^qG(y)dy\Bigl)^\frac{1}{q}} \le
      \frac{1}{c_E(u)^\frac{n-p}{np}}
      \Bigr(\int_MJ_E^{-\frac{p-1}{n-1}}|\nabla u|^p\Bigl)^\frac{1}{p} \\
      + \frac{n(n-p)}{p(n-1)c_E(u)^\frac{n-p}{np}}
      \Bigr(\int_MJ_E^{-\frac{p-1}{n-1}}|u|^p|H|^p\Bigl)^\frac{1}{p}.
  \end{multline*}
  Taking $G=v^\frac{np}{n-p}$ where $\|v\|_\frac{np}{n-p}=1$, the supremum over
  all function $v$ of the left handside is the Sobolev constant of $\Rn$ (cf.
  the caracterisation of $S_{n,p}$ given in section 1). Thus we have
  \begin{multline*}
    S_{n,p}\Bigl(\int_MJ_E^\frac{1}{n-1}|u|^\frac{np}{n-p}dv_M\Bigr)^\frac{n-p}{np}
      \le \int_MJ_E^{-\frac{p-1}{n-1}}|\nabla u|^pdv_M \\
     + \frac{n(n-p)}{p(n-1)}\int_MJ_E^{-\frac{p-1}{n-1}}|H||u|dv_M.
  \end{multline*}
  Moreover, this inequality is sharp because it is just the Euclidean $L^p$
  Sobolev inequality of $\Rn$ when $M=E$.
\end{proof}

%
%
%
%
\section{Inequalities for submanifolds in warped products}

In the previous section the main tools where the projection on
a subspace (seen as the Euclidean $n$ space) and the use of optimal
transport in this subspace. As
soon as we have these tools on a manifold, we can expect Sobolev
inequalities for its submanifolds.

A typical example is the hyperbolic space, where horospheres are
isometric to the Euclidean space and where the projections on them
are well defined. In fact, the Hyperbolic space is a particular case
of warped product for which we can use optimal transportation to get
weighted Sobolev inequalities on their submanifolds.

\subsubsection*{Warped products}
Consider a warped product $N=\R\times\Rnk$ (with $k\ge 0$)
endowed with the metric $g_N=dt^2+w(t)^2dy^2$ where $w$ is a smooth
function, and $dy^2$ is the Euclidean metric on $\Rnk$. In the sequel
we shall note $(t,y)$ a point in $N$ where $y=(y_1,\dots,y_{n+k})\in\Rnk$.

Let $E$ be a $n$-linear subspace of $\Rnk$ ; we can assume, without loss
of generality, that $E$ is the subspace spanned by the first $n$ vectors
of the canonical basis of $\Rnk$. We denote by $p:N\to E$ the projection
on $E$ : $p(t,y)=(y_1,\dots,y_n)$. In the sequel we assume that $E$ is
endowed with the Euclidean metric, and we have that, if $\xi\in T_{(t,y)}N$
belongs to the subspace spanned by
$(\frac{\partial}{\partial y_1},\dots,\frac{\partial}{\partial y_n})$
then $|T_{(t,y)}p.\xi|=\frac{1}{w(t)}|\xi|$.

Let $V:E\to\R$ be a function on $E$, and let $\bV$ be its extension to $N$
defined by $\bV(t,y)=V(y_1,\dots,y_n)$. By a standard
computation we have $w(t)|\bnabla\bV(t,y)|=|\nabla V(p(t,y))|$ and
\begin{equation}
  \bhess\bV = -2\frac{w'}{w}\sum_{i=1}^{n}\dd{V}{y_i} dy_idt
   + \sum_{i,j=1}^{n}\frac{\partial^2V}{\partial y_i\partial y_j}dy_idy_j.
   \label{eqn-hessian}
\end{equation}
The main difference with the Euclidean case, is that, with the terms coming from
the Hessian of $V$, we get extra terms coming from the extrinsic curvature of
$\{t\}\times\Rnk$ in $N$.

Consider now an isometric immersion $i:M^n\to N$, where $M$ is an
$n$-dimensional manifold, and let $\tau:M\to\R$ be the restriction to
$M$ of the first coordinate function on $N$.

For $x\in M$, let $J_E(x)=|\det(q)|$, where $q$ is the
orthogonal projection (in $T_xN$) from $T_xM$ to the subspace
spanned by $(\frac{\partial}{\partial y_1},\dots,\frac{\partial}{\partial y_n})$.
If we still note $p:M\to E$ the restriction of the projection $p$ to
the submanifold $M$, for each $x\in M$ the absolute value of the Jacobian
determinant of $p$ at $x$ is $\frac{1}{w(\tau(x))^n}J_E(x)$. The critical set
of $p$ is $\cC=\{x\in M\ |\ J_E(x)=0\}$.

Considering a convex function $V$ on $E$, we have that the symetric two form
$$
\cB=\bhess\bV + 2\frac{w'}{w}\sum_{i=1}^{n}\dd{V}{y_i} dy_idt
$$
is nonnegative, and for any $x\in T_xM$ we have
\begin{equation}
	\tr(\cB_{|_{T_xM}}) = \tr(\bhess\bV_{|_{T_xM}})
		+ 2\frac{w'}{w}\langle\nabla\tau,\nabla V_M\rangle.
	\label{eqn-trace_B}
\end{equation}

In the sequel we will use a nonsmooth convex function $V$ on $E$. Using its
second derivatives (well defined almost everywhere) and equation \ref{eqn-hessian},
we define the Hessian of $\bV$ in the sense of Aleksandrov :
\begin{equation}
   \bhess_A\bV = -2\frac{w'}{w}\sum_{i=1}^{n}\dd{V}{y_i} dy_idt
   + \sum_{i,j=1}^{n}\frac{\partial^2V}{\partial y_i\partial y_j}dy_idy_j.
   \label{eqn-hessian_aleksandrov}
\end{equation}
Moreover we can mimic the proof of proposition \ref{prop-cvxe_submfd},
the main point for doing this being that the Riemannian measure of $N$ is
a product measure which can be written using the measure on $M$ and
the Jacobian determinant of $p$. Let $V_M$
be the restriction of $\bV$ to $M$, using Riesz theorem together with
equations \ref{eqn-laplacian_submfd} and \ref{eqn-trace_B} we have that
$\Delta_\Dp V_M -n\langle H,\bnabla\bV\rangle + 2\frac{w'}{w}\langle\nabla\tau,\nabla V_M\rangle$
is a nonnegative Radon measure $\nu$. Therefore, $\Delta_\Dp V_M$ is also
a Radon measure and, if we note $\Delta_AV_M$ the density of its absolutely
continuous part with respect to $dv_M$, we have
$$
\Delta_AV_M = \tr(\bhess_A\bV(x)_{|_{T_xM}}) + n\langle H,\bnabla\bV \rangle
$$
on any domain $D\subset M$ on which $p$ is a local diffeomorphism.

The other consequence of the nonnegativity of $\nu$ is that, mimicking
the arguments leading to inequality \ref{eqn-step4}, we have
\begin{multline}
  \int_{M\setminus\cC}\varphi\Delta_AV_M - n\varphi\langle H,\bnabla\bV\rangle
    +2\varphi\frac{w'}{w}\langle\nabla\tau,\nabla V_M\rangle  
    \le -\int_M\langle\nabla \varphi,\nabla V_M\rangle dv_M \\
    + n\int_M\varphi\frac{|H|}{w} + 2\varphi\frac{w'}{w}\langle\nabla\tau,\nabla V_M\rangle
  \label{eqn-int_part_warp}
\end{multline}
for any nonnegative $\varphi\in C_c^\infty(M)$.

\begin{rem}
  When taking $w(t)=\ex^t$, the manifold $N$ is isometric to the hyperbolic
  space $\Hnk$. In this case, the first coordinate function is a Buseman function
  centered at some point at infinity, the submanifolds $\{t\}\times\Rnk$ are
  horospheres, and the metric $g_N=dt^2+\ex^{2t}dy^2$ is the hyperbolic metric
  read in horospherical coordinates.
\end{rem}

\subsubsection*{Weighted isoperimetric inequality}
Using the notations above we get the following result :
\begin{thm}
  Let $i:M^n\to\R\times\Rnk$ be an isometric immersion where $\R\times\Rnk$ is endowed
  with the metric $dt^2+w(t)^2dy^2$, and let $E$ be a $n$-dimensional
  linear subspace of $\Rnk$. For any regular domain $\Omega\subset M$ we have
  $$
  n\omega_n^\frac{1}{n}\Bigl(\int_\Omega \bigl(w(\tau)^nJ_E\bigr)^\frac{1}{n-1}dv_M\Bigr)^\frac{n-1}{n}
   \le \int_{\partial\Omega}w(\tau)dv_{\partial\Omega} + n\int_\Omega w(\tau)|H|dv_M.
  $$

  The Sobolev counterpart of this inequality is
  \begin{multline*}
  n\omega_n^\frac{1}{n}\Bigl(\int_M\bigl(w(\tau)^nJ_E\bigr)^\frac{1}{n-1}
    |u|^\frac{n}{n-1}dv_M\Bigr)^\frac{n-1}{n}
    \le \int_M w(\tau)|\nabla u|dv_M \\ + n\int_M w(\tau)|H||u|dv_M
  \end{multline*}
  for any function $u\in C_c^\infty(M)$.\label{thm-warpisop}
\end{thm}
\begin{proof}
  Let $f=\frac{(w(\tau)^nJ_E)^\frac{1}{n-1}u^\frac{n}{n-1}}{c_E(u)}$,
  where $u\in C_c^\infty(M)$ is a nonnegative function, and
  $c_E(u)=\int_M(w(\tau)^nJ_E)^\frac{1}{n-1}u^\frac{n}{n-1}dv_M$.

  Following the proof of theorem \ref{thm-sobL1}, there exist a convex function
  $V$ on $E$ such that $\nabla V$ is the solution of the problem of Monge
  between $p_\#\mu$ and $\frac{\chi_{B_E}}{\omega_n}dz$. 
  
  As $f$ vanishes on the critical set $\cC$, the measure $p_\#\mu$ is absolutely
  continuous with respect to th Lebesgue measure on $E$ and its density reads
  $$
  F(y) = \sum_{x\in p^{-1}(y)\cap\supp(\mu)}\frac{w(\tau(x))^nf(x)}{J_E(x)}.
  $$
  
  Also $V$ may not be smooth, we can use derivatives in the sense of Aleksandrov
  and, by a change of variable in $E$, we get
  \begin{equation}
    \omega_n\frac{w(\tau(x))^nf(x)}{J_E(x)} \le \omega_nF(p(x)) = \det(\hess_AV(p(x)))
    \label{eqn-trans}
  \end{equation}
  for a.a. $x\in\supp(\mu)$.
  Let $\cB=\bhess_A\bV + 2\frac{w'}{w}\sum_{i=1}^{n}\dd{V}{y_i} dy_idt$ ;
  for any unitary vector $\xi\in T_xM$, using
  equation \ref{eqn-hessian_aleksandrov}, we have
  $\cB(\xi,\xi) = \hess_AV(p(x))(T_xp.\xi,T_xp.\xi)$. Therefore we get
  $$
  \frac{J_E(x)^2}{w(\tau(x))^{2n}}\det(\hess_AV(p(x))) = \det(\cB_{|_{T_xM}}).
  $$

  It follows from equation \ref{eqn-hessian_aleksandrov} that, as $V$ is convex,
  $\cB_{|_{T_xM}}$ is nonnegative and, with the geometric-arithmetic
  inequality, the inequality \ref{eqn-trans} becomes
  \begin{equation}
  	n\omega_n^\frac{1}{n}\frac{J_E^\frac{1}{n}}{w(\tau)}f^\frac{1}{n} \le \tr(\cB_{|_{T_xM}}),
    \label{eqn-J_Ef}
  \end{equation}
  from which we have
  \begin{equation}
    n\omega_n^\frac{1}{n}\frac{J_E^\frac{1}{n}}{w(\tau)}f^\frac{1}{n} \le
      \Delta_AV_M - n\langle H,\bnabla\bV\rangle
  	  + 2\frac{w'(\tau)}{w(\tau)}\langle \nabla\tau,\nabla V_M\rangle,
    \label{eqn-G_and_f}
  \end{equation}
  and multiplying by $u(x)w(\tau(x))^2$ gives
  \begin{multline*}
  \frac{n\omega_n^\frac{1}{n}}{c_E(u)^\frac{1}{n}}
    (w(\tau)^nJ_E)^\frac{1}{n-1}u^\frac{n}{n-1} \le
    uw(\tau)^2\Delta_AV_M - nuw(\tau)^2\langle H,\bnabla\bV\rangle \\
    + 2uw(\tau)w'(\tau)\langle \nabla \tau,\nabla V_M \rangle.
  \end{multline*}
  Integrating this inequality on $M\setminus\cC$ and using
  inequality \ref{eqn-int_part_warp} we get
  \begin{multline*}
  \frac{n\omega_n^\frac{1}{n}}{c_E(u)^\frac{1}{n}}
   \int_M(w(\tau)^nJ_E)^\frac{1}{n-1}u^\frac{n}{n-1}dv_M
     \le -\int_Mw(\tau)^2\langle \nabla u,\nabla V_M\rangle \\
     + n\int_Muw(\tau)|H|,
  \end{multline*}
  and since $w(\tau)|\nabla V_M|\le w(\tau)|\bnabla\bV|\le 1$ we obtain
  the desired inequality :
  $$
  \frac{n\omega_n^\frac{1}{n}}{c_E(u)^\frac{1}{n}}
   \int_M(w(\tau)^nJ_E)^\frac{1}{n-1}u^\frac{n}{n-1}dv_M
     \le \int_Mw(\tau)|\nabla u| + n\int_Muw(\tau)|H|.
  $$
\end{proof}

\subsubsection*{Weighted $L^p$ Sobolev inequalities}
As for the Euclidean submanifolds, we can also prove weighted $L^p$ Sobolev inequalities.
\begin{thm}
  Let $i:M^n\to\R\times\Rnk$ be an isometric immersion where $\R\times\Rnk$ is endowed
  with the metric $dt^2+w(t)^2dy^2$, and let $E$ be a $n$-dimensional
  linear subspace of $\Rnk$ such that the critical set of the projection
  on $E$ is negligible in $M$. For any $1< p <n$, and for any function $u\in C_c^\infty(M)$ we have
  \begin{multline*}
  S_{n,p}\Bigl(\int_M\bigl(w(\tau)^nJ_E\bigr)^\frac{1}{n-1}|u|^\frac{np}{n-p}dv_M\Bigr)^\frac{n-p}{np}
   \le \int_MJ_E^{-\frac{p-1}{n-1}}w(\tau)^\frac{n-p}{n-1}|\nabla u|^pdv_M \\
   + \frac{n(n-p)}{p(n-1)}\int_MJ_E^{-\frac{p-1}{n-1}}w(\tau)^\frac{n-p}{n-1}|H|^p|u|^pdv_M.
  \end{multline*}
\end{thm}
\begin{proof}[Sketch of proof]
  Let us start with a function $u\in C_c^\infty(M)$ and the measure $\mu=fdv_m$ where
  $f=\frac{(w(\tau)^nJ_E)^\frac{1}{n-1}u^\frac{np}{n-p}}{c_E(u)}$, with
  $c_E(u)=\int_M(w(\tau)^nJ_E)^\frac{1}{n-1}u^\frac{np}{n-p}$.
  
  Then we just have to follow step by step the proof of theorem \ref{thm-SobLp}, using
  the tools of the proof of theorem \ref{thm-warpisop} to handle the different terms
  coming from the metric of $N$.
\end{proof}

%
%
%
%

\end{document}